\documentclass{amsart}
\usepackage{amssymb,amstext,amsmath,amscd,amsthm,amsfonts,enumerate,graphicx,latexsym,stmaryrd,multicol}
\usepackage[usenames]{color}
\usepackage{setspace}
\usepackage{bm}
\usepackage[all]{xy}
\newtheorem{thm}{Theorem}[section]
\newtheorem{lem}[thm]{Lemma}
\newtheorem{prop}[thm]{Proposition}
\newtheorem{cor}[thm]{Corollary}
\theoremstyle{definition}
\newtheorem{dfn}[thm]{Definition}

\newtheorem{ex}[thm]{Example}

\newtheorem{rmk}[thm]{Remark}

\theoremstyle{remark}

\newtheorem*{ac}{Acknowlegments}

\newtheorem*{proof of claim}{Proof of Claim}
\numberwithin{equation}{thm}
\def\ass{\operatorname{Ass}}
\def\cm{\mathsf{CM}}

\def\depth{\operatorname{depth}}
\def\Ext{\operatorname{Ext}}

\def\ge{\geqslant}
\def\height{\operatorname{ht}}
\def\Hom{\operatorname{Hom}}

\def\ker{\operatorname{Ker}}
\def\le{\leqslant}

\def\m{\mathfrak{m}}

\def\P{\mathbb{P}}
\def\p{\mathfrak{p}}

\def\q{\mathfrak{q}}

\def\spec{\operatorname{Spec}}
\def\supp{\operatorname{Supp}}

\def\V{\mathrm{V}}

\def\mcm{\mathsf{MCM}}
\def\gor{\mathsf{Gor}}
\def\fid{\mathsf{FID}}
\def\free{\mathsf{Free}}
\def\id{\operatorname{id}}
\def\D{\mathrm{D}}
\def\S{\mathsf{S}}
\def\T{\mathsf{T}}

\begin{document}
\allowdisplaybreaks
\title{Openness of various loci over Noetherian rings}
\author{Kaito Kimura}
\address[KK]{Graduate School of Mathematics, Nagoya University, Furocho, Chikusaku, Nagoya 464-8602, Japan}
\email{m21018b@math.nagoya-u.ac.jp}
\begin{abstract}
In this paper, we consider the openness of the $\P$-locus of a finitely generated module over a commutative noetherian ring in the case where $\P$ is each of the properties $\fid, \gor, \cm, \mcm$, $(\S_n)$, and $(\T_n)$.
One of the main results asserts that $\fid$-loci over an acceptable ring are open.
We give a module version of the Nagata criterion, and prove that it holds for all of the aforementioned properties.
\end{abstract}
\maketitle
%%%%%%%%%%%%%%%%%%%%%%%%%%%%%%%%%%%%%%%%%%%%%%%%%%%%%%%%%%%%
\section{Introduction}

Throughout the present paper, all rings are assumed to be commutative and noetherian, and all modules be finitely generated.

Let $\P$ be a property of local rings, and $R$ a ring.
The set of prime ideals $\p$ of $R$ such that the local ring $R_\p$ satisfies $\P$ is called the $\P$-locus of $R$.
For example, the regular, complete intersection, Gorenstein, and Cohen--Macaulay properties and Serre's condition $(\S_n)$ can be considered as $\P$.
When $\P$ is any one of the properties appearing above, the $\P$-locus is stable under generalization.
Therefore, it is a natural question to ask when the $\P$-locus is open in the Zariski topology for a given $\P$.
This question has been studied for a long time by many people.
Nagata \cite{N} produced the following condition, which is called the \textit{Nagata criterion}:
\begin{enumerate}[\rm(NC):]
\item if the $\P$-locus of $R/\p$ contains a nonempty open subset of $\spec (R/\p)$ for all prime ideals $\p$ of $R$, then the $\P$-locus of $R$ is an open subset of $\spec (R)$.
\end{enumerate}
\noindent The Nagata criterion holds for the regular, complete intersection, Gorenstein, and Cohen--Macaulay properties and Serre's condition $(\S_n)$; see \cite{GM, Mat, N, Ta}.

Let $\P$ be a property of modules over a local ring, and $M$ an $R$-module.
The set of prime ideals $\p$ of $R$ such that the module $M_\p$ over the local ring $R_\p$ satisfies $\P$ is called the $\P$-locus of $M$ (over $R$).
The same question can be asked for the $\P$-locus of $M$ which is stable under generalization, and by the so-called topological Nagata criterion, the $\P$-locus of $M$ is open if and only if the $\P$-locus of $M$ contains a nonempty open subset of $\V(\p)$ for all $\p$ belonging to the $\P$-locus of $M$.
It is well-known fact that the free locus \cite{Mat} and the finite projective dimension locus \cite{BS} are always open.
The Cohen--Macaulay locus of a module over an excellent ring is open \cite{G}.
Furthermore, the Gorenstein locus of a module over an acceptable ring in the sense of Sharp \cite{Sh} is open \cite{L}, and so is the finite injective dimension locus of a module over an excellent ring \cite{T}.

In this paper, we consider the openness of the $\P$-locus of a module in the case where $\P$ is each of the finite injective dimension property ($\fid$), the Gorenstein property ($\gor$), the Cohen--Macaulay property ($\cm$), the maximal Cohen--Macaulay property ($\mcm$), and Serre's conditions $(\S_n)$ and $(\T_n)$; for the definition of $(\T_n)$ see Definition \ref{def}.
First of all, we handle the case $\P=\fid$.
For a fixed element $\p$ of the $\P$-locus of $M$, we give a (necessary and) sufficient condition for the $\P$-locus of $M$ to contain a nonempty open subset of $\V(\p)$.
Also, we study some rings that satisfy those conditions for all prime ideals.
The main result in this direction is the following theorem concerning $\fid$-loci over an acceptable ring.
\begin{thm}\label{main1}
The $\fid$-locus of a module over an acceptable ring is open in the Zariski topology.
In particular, the $\fid$-locus of a module over a homomorphic image of a Gorenstein ring is open.
\end{thm}
\noindent Next, we confirm that some results on the finite injective dimension property hold on other properties.
We give a module version of the Nagata criterion for a property $\P$ of modules over a local ring:
\begin{enumerate}[\rm(NC)$^\ast$:]
\item if the $\P$-locus of $M/\p M$ over $R/\p$ contains a nonempty open subset of $\spec (R/\p)$ for all prime ideals $\p$ of $R$ belonging to $\supp_R(M)$, 
then the $\P$-locus of $M$ over $R$ is an open subset of $\spec (R)$.
\end{enumerate}
\noindent It is natural to ask for which properties (NC)$^\ast$ hold.
We prove the following theorem.
\begin{thm}\label{main2}
Let $n\ge 0$ be an integer.
Then {\rm(NC)$^{\ast}$} holds for each $\P\in\{\fid, \gor, \cm, \mcm, (\S_n), (\T_n)\}$.
\end{thm}

The organization of this paper is as follows. 
Section 2 is denoted to preliminaries for the later sections.
In Sections 3 and 4, we study the openness of the $\fid$-locus of a module.
We give a sufficient condition for the $\fid$-locus of a module to be open, and obtain Theorem \ref{main1}.
In Section 5, we observe that the same results as we gave in the previous sections hold for the Cohen--Macaulay property.
In Section 6, we consider when the $(\S_n)$ and $(\T_n)$-loci of a module are open.
It is seen that if the $\mcm$-locus (resp. the $\cm$-locus) is open, then so is the $(\S_n)$-locus (resp. the $(\T_n)$-locus).
In Section 7, we prove Theorem \ref{main2}.

%%%%%%%%%%%%%%%%%%%%%%%%%%%%%%%%%%%%%%%%%%%%%%%%%%%%%%%%%%%%%%%%%%%%
\section{Notation and lemmas}

In this section, we state the definitions of notions used in this paper, and give basic lemmas about the Zariski topology.
Throughout the present paper, let $R$ be a ring.
Let $\P$ be a property of local rings.
The subset $\P(R)=\{\p\in\spec (R)\mid \P$ holds for $R_\p\}$ of $\spec (R)$ is called the $\P$-locus of $R$.
Similarly, if $\P$ is a property of modules over a local ring, 
then $\P_R (M)=\{\p\in\spec (R)\mid \P$ holds for $M_\p\}$ is called the $\P$-locus of $M$ (over $R$) for an $R$-module $M$.

\begin{dfn}\label{def}
Let $(R,\m,k)$ be a local ring, $M$ an $R$-module, and $n$ an integer.
\begin{itemize}
\item If $\dim M \le \depth M$, then $M$ is called a \textit{Cohen-Macaulay module}.
\item If $\dim R \le \depth M$, then $M$ is called a \textit{maximal Cohen-Macaulay module}.
\item An $R$-module $M$ is a \textit{Gorenstein module} (\textit{of typr} $r$) if 
$$
\dim_k\Ext_R^i(k,M)=
\begin{cases}
r&(i=\dim R),\\
0&(i\ne \dim R).
\end{cases}
$$
\item A Gorenstein module of type $1$ is called a \textit{canonical module}.
\end{itemize}
In general, let $R$ be a (not necessarily local) ring, and $M$ an $R$-module.
We denote by $\operatorname{Ann}_R (M)$ the annihilator ideal of $M$.
The injective dimension of $M$ is denoted by $\id_{R}M$.
We say that $M$ is a Cohen-Macaulay module (resp. a maximal Cohen-Macaulay module, a Gorenstein module, and a canonical module) if so is the module $M_\m$ over the local ring $R_\m$ for every maximal ideal $\m$ of $R$. 
If $R$ is itself a Cohen-Macaulay (resp. Gorenstein) module, then it is called a Cohen-Macaulay (resp. Gorenstein) ring.
We say that
\begin{itemize}
\item $M$ satisfies Serre's condition $(\S_n)$ if $\depth M_\p \ge {\rm inf}\{n, \height\p \}$ for all $\p\in\spec (R)$, and
\item $M$ satisfies Serre's condition $(\T_n)$ if $\depth M_\p \ge {\rm inf}\{n, \dim M_\p \}$ for all $\p\in\spec (R)$.
\end{itemize}
\end{dfn}

In this paper, the following notation is used.

\begin{dfn}
Let $M$ be an $R$-module, $I$ an ideal of $R$, $f$ an element of $R$, and $n$ an integer.
\begin{itemize}
\item $\D(f) = \{\p\in\spec (R)\mid f\notin\p\}$.
\item $\D(I) = \{\p\in\spec (R)\mid I\nsubseteq\p\}$.
\item $\V(I) = \{\p\in\spec (R)\mid I\subseteq\p\}$.
\item $\supp_R (M) = \{\p\in\spec (R)\mid$ The $R_\p$-module $M_\p$ is nonzero$ \}$.
\item $\free_R (M) = \{\p\in\spec (R)\mid$ The $R_\p$-module $M_\p$ is free$ \}$.
\item $\cm(R) = \{\p\in\spec (R)\mid$ The local ring $R_\p$ is Cohen--Macaulay$ \}$.
\item $\gor(R) = \{\p\in\spec (R)\mid$ The local ring $R_\p$ is Gorenstein$ \}$.
\item $\S_n(R) = \{\p\in\spec (R)\mid$ The $R_\p$-module $R_\p$ satisfies $(\S_n) \}$.
\item $\cm_R (M) = \{\p\in\spec (R)\mid$ The $R_\p$-module $M_\p$ is Cohen--Macaulay$ \}$.
\item $\mcm_R (M) = \{\p\in\spec (R)\mid$ The $R_\p$-module $M_\p$ is maximal Cohen--Macaulay$ \}$.
\item $\fid_R (M) = \{\p\in\spec (R)\mid$ The $R_\p$-module $M_\p$ has finite injective dimension$ \}$.
\item $\gor_R (M) = \{\p\in\spec (R)\mid$ The $R_\p$-module $M_\p$ is Gorenstein$ \}$.
\item $\S_n^R(M) = \{\p\in\spec (R)\mid$ The $R_\p$-module $M_\p$ satisfies $(\S_n) \}$.
\item $\T_n^R(M) = \{\p\in\spec (R)\mid$ The $R_\p$-module $M_\p$ satisfies $(\T_n) \}$.
\end{itemize}
\end{dfn}

\begin{rmk}\label{rmk1}
Note that the zero module is thought of as Cohen-Macaulay, maximal Cohen-Macaulay, and Gorenstein.
It is also considered to satisfy $(\S_n)$ and $(\T_n)$.
We see that an $R$-module $M$ is a Gorenstein module if and only if it is maximal Cohen--Macaulay and has finite injective dimension; see \cite[Theorem 1.2.8, Proposition 3.1.14, and Theorem 3.1.17]{BH}.
Hence, we have $\gor_R (M)=\fid_R (M)\cap\mcm_R (M)$.
\end{rmk}

From now on, $\spec (R)$ will be considered to be with the Zariski topology, and any subset of $\spec (R)$ will be considered to be with the subspace topology.
Below is called the \textit{topological Nagata criterion}.

\begin{lem}\cite[Theorem 24.2]{Mat}\label{Mat24.2}
Let $U$ be a subset of $\spec (R)$.
Then $U$ is open if and only if the following two statements hold true.
\begin{enumerate}[\rm(1)]
\item $U$ is stable under generalization, that is, if $\p\in U$ and $\q\in \spec(R)$ with $\q\subseteq \p$, then $\q\in U$. 
\item $U$ contains a nonempty open subset of $\V(\p)$ for all $\p\in U$.
\end{enumerate}
\end{lem}

Note that $\cm_R (M)$, $\mcm_R (M)$, $\fid_R (M)$, $\gor_R (M)$, $\S_n^R(M)$, and $\T_n^R(M)$ are stable under generalization for any $R$-module $M$.
Therefore, in order to show that each of these subsets is open, it suffices to verify that it satisfies (2) in the above lemma.
The following two lemmas are useful for that.

\begin{lem}\label{replace}
Let $\p$ be a prime ideal of $R$, and $f\in R\setminus\p$.
Let $F:\D(f)\to \spec (R_f)$ be the natural homeomorphism, $S$ a subset of $\spec (R)$, and $T$ the image of $\D(f)\cap S$ by $F$.
Then $S$ contains a nonempty open subset of $\V(\p)$ if and only if $T$ contains a nonempty open subset of $\V(\p R_f)$.
\end{lem}

\begin{proof}
Suppose that there exists an open subset $U$ of $\spec (R)$ such that $\V(\p)\cap U$ is nonempty and is contained in $S$.
There exists a prime ideal $\q$ of $R$, which belongs to $U$ and contains $\p$.
Since $U$ is open, $\p$ is in $U$.
Hence $\D(f)\cap \V(\p)\cap U$ is a nonempty open subset of $\D(f)\cap \V(\p)$.
Note that $F$ induces a homeomorphism between $\D(f)\cap \V(\p)$ and $\V(\p R_f)$.
We see that the subset $F(\D(f)\cap \V(\p)\cap U)$ of $F(\D(f)\cap S)=T$ is a nonempty open subset of $\V(\p R_f)$.

Conversely, if there exists a nonempty open subset $U$ of $\V(\p R_f)$ that is contained in $T$, then the subset $F^{-1}(U)$ of $F^{-1}(T)=\D(f)\cap S$ is a nonempty open subset of $\D(f)\cap \V(\p)$.
As $\D(f)$ is an open subset of $\spec (R)$, $F^{-1}(U)$ is also an open subset of $\V(\p)$.
\end{proof}

\begin{lem}\label{replace2}
Let $I$ be an ideal of $R$, and $\p\in\V(I)$.
Let $F:\V(I)\to \spec (R/I)$ be the natural homeomorphism, $S$ a subset of $\spec (R)$, and $T$ the image of $\V(I)\cap S$ by $F$.
Then $S$ contains a nonempty open subset of $\V(\p)$ if and only if $T$ contains a nonempty open subset of $\V(\p/I)$.
\end{lem}

\begin{proof}
Note that $F$ induces a homeomorphism between $\V(\p)$ and $\V(\p/I)$.
Suppose that $U$ is a nonempty open subset of $\V(\p)$ that is contained in $S$.
Then $F(U)$ is a nonempty open subset of $\V(\p/I)$ and is contained in $F(\V(I)\cap S)=T$.
Conversely, if $U$ is a nonempty open subset of $\V(\p/I)$ that is contained in $T$, then $F^{-1}(U)$ is a nonempty open subset of $\V(\p)$ and is contained in $F^{-1}(T)=\V(I)\cap S$.
\end{proof}

We prepare an elementary lemma.

\begin{lem}\label{base}
Let $M$ be an $R$-module, and let $\p$ be a prime ideal of $R$.
\begin{enumerate}[\rm(1)]
\item $\p$ belongs to any nonempty open subset of $\V(\p)$.
In particular, the zero ideal of $R/\p$ belongs to any nonempty open subset of $\spec (R/\p)$.
\item If $M_\p=0$, then $M_f=0$ for some $f\in R\setminus\p$.
\item Suppose that a sequence $\bm{x}=x_1,\ldots,x_n$ of elements in $\p$ is an $M_\p$-regular sequence.
Then there exists $f\in R\setminus\p$ such that $\bm{x}$ is an $M_f$-regular sequence.
\item If $\p$ is a minimal prime ideal of an ideal $I$, then $\sqrt{I R_f}=\p R_f$ for some $f\in R\setminus\p$.
\item  If $R$ is an integral domain, then $M_f$ is a free $R_f$-module for some $f\in R\setminus\{0\}$.
\item Suppose that $\p^r M=0$ for some $r>0$.
Then there exists $f\in R\setminus\p$ such that $(\p^{i-1} M/\p^i M)_f$ is a free $(R/\p)_f$-module for each $1\le i\le r$.
\end{enumerate}
\end{lem}

\begin{proof}
(1): An analogous argument to the former part of the proof of Lemma \ref{replace} shows the assertion.

(2): There is an element $f\in\operatorname{Ann}_R (M)\setminus\p$ since $\p$ is not in $\supp_R (M)$.
We obtain $M_f=0$.

(3): We may assume $n=1$.
Let $\phi$ be the multiplication map of $M$ by $x_1$.
We have $(\ker \phi)_\p=0$ because $x_1$ is an $M_\p$-regular element.
It follows from (2) that $(\ker \phi)_f=0$ for some $f\in R\setminus\p$.
Then $x_1$ is an $M_f$-regular element.
Note that $M_f/x_1 M_f\ne 0$ since $\p$ belongs to $\supp_R (M_f/x_1 M_f)$.

(4): Let $\p=\p_1, \p_2, \ldots, \p_n$ be all the minimal prime ideals of $I$.
Since $\p_i\nsubseteq\p$ for any $2\le i\le n$, there exists $f\in R\setminus\p$ such that $f\in\p_i$ for each $2\le i\le n$.
We easily obtain $\sqrt{I R_f}=\p R_f$.

(5) and (6): The assertions follow from \cite[Theorem 24.1]{Mat}.
\end{proof}

\begin{rmk}\label{rmk2}
Let $M$ be an $R$-module, $\p$ a prime ideal of $R$, and $\P\in\{\fid, \gor, \cm, \mcm, (\S_n), (\T_n)\}$.
It is well-known fact that $\spec (R)\setminus\supp_R (M)$ is an open subset of $\spec (R)$.
If $\p$ belongs to $\spec (R)\setminus\supp_R (M)$, then $\V(\p)\cap(\spec (R)\setminus\supp_R(M))$ is a nonempty open subset of $\V(\p)$, and it is contained in the $\P$-locus of $M$; see Remark \ref{rmk1}.
\end{rmk}

%%%%%%%%%%%%%%%%%%%%%%%%%%%%%%%%%%%%%%%%%%%%%%%%%%%%%%%%%%%%%%%%%%%%
%%%%%%%%%%%%%%%%%%%%%%%%%%%%%%%%%%%%%%%%%%%%%%%%%%%%%%%%%%%%%%%%%%%%
\section{The openness of the $\fid$-locus of a module}

In this section, we study the openness of the $\fid$-locus of a module by relating it to the $\gor$-locus of a ring.
More precisely, we examine the relationship between the hypothesis of the Nagata criterion for the Gorenstein property and the openness of the finite injective dimension locus.
Recall a few definitions that will be used in the following; see \cite{Sha, Sh}.

\begin{dfn}
Let $S$ be a ring.
A ring homomorphism $\phi:R\to S$ is said to be a \textit{Gorenstein homomorphism} if $\phi$ is flat and all the fiber rings of $\phi$ are Gorenstein.
\end{dfn}

\begin{dfn}
A ring $R$ is said to be \textit{well-fibered} if the natural ring homomorphism from $R_\p$ to its completion is Gorenstein for every prime ideal $\p$ of $R$.
\end{dfn}

\begin{dfn}
A ring $R$ is said to be acceptable if the following three conditions are satisfied.
\begin{enumerate}[\rm(1)]
\item $R$ is universally catenary.
\item For all finitely generated $R$-algebras $S$, $\gor(S)$ is open.
\item $R$ is well-fibered.
\end{enumerate}
\end{dfn}

\begin{ex}
\begin{enumerate}[\rm(1)]
\item An Artinian ring is well-fibered.
\item Let $S=R[X_1, \ldots X_n]$ be a polynomial ring over $R$. 
The natural ring homomorphism $\phi:R \to S$ is Gorenstein since all the fiber rings of $\phi$ are regular.
\item Any homomorphic image of a Gorenstein ring is acceptable; see \cite{GM, Sh}.
\end{enumerate}
\end{ex}

The key role is played by the lemma below.

\begin{lem}\cite[Proposition 2.4]{T}\label{T 2.4}
Let $M$ be an $R$-module, and let $\p\in\fid_R (M)$.
Suppose that $\fid_{R/\p} (\Ext_R^j(R/\p,M))$ contains a nonempty open subset of $\spec (R/\p)$ for each integer $j$ with $0\le j\le \height\p$.
Then $\fid_{R} (M)$ contains a nonempty open subset of $\V(\p)$.
\end{lem}

The main result of this section is the following theorem, whose proof uses the above lemma.

\begin{thm}\label{thmA}
Let $M$ be an $R$-module, and let $\p\in\fid_R (M)$.
Suppose that $\gor(R/\p)$ contains a nonempty open subset of $\spec (R/\p)$.
Then $\fid_{R} (M)$ contains a nonempty open subset of $\V(\p)$.
\end{thm}

\begin{proof}
Fix an integer $0\le j\le \height\p$, and let $N=\Ext_R^j(R/\p,M)$.
The locus $\free_{R/\p}(N)$ is an open subset of $\spec (R/\p)$ by \cite[Theorem 4.10 (ii)]{Mat}.
The zero ideal of $R/\p$ belongs to $\free_{R/\p}(N)$ since $R/\p$ is an integral domain.
By assumption, there is a nonempty open subset $U$ of $\spec (R/\p)$, which is contained in $\gor(R/\p)$.
It follows from Lemma \ref{base} (1) that $U\cap\free_{R/\p}(N)$ is a nonempty open subset of $\spec (R/\p)$.
We have $U\cap\free_{R/\p}(N)\subseteq\gor(R/\p)\cap\free_{R/\p}(N)\subseteq\fid_{R/\p} (N)$.
Lemma \ref{T 2.4} implies that $\fid_{R} (M)$ contains a nonempty open subset of $\V(\p)$.
\end{proof}

The result below can be obtained from Theorem \ref{thmA}.

\begin{cor}\label{corA}
\begin{enumerate}[\rm(1)]
\item Let $M$ be an $R$-module.
Suppose that $\gor(R/\p)$ contains a nonempty open subset of $\spec (R/\p)$ for any $\p\in\fid_R (M)\cap\supp_R(M)$.
Then $\fid_{R} (M)$ is an open subset of $\spec (R)$.
\item Suppose that $\gor(R/\p)$ contains a nonempty open subset of $\spec (R/\p)$ for all prime ideals $\p$ of $R$.
Then $\fid_{R} (M)$ is an open subset of $\spec (R)$ for any $R$-module $M$. 
\end{enumerate}
\end{cor}

\begin{proof}
(1): The locus $\fid_{R} (M)$ clearly satisfies the condition (1) in Lemma \ref{Mat24.2}.
Let $\p\in\fid_R (M)$.
It follows from Remark \ref{rmk2} and Theorem \ref{thmA} that $\fid_{R} (M)$ contains a nonempty open subset of $\V(\p)$.
Therefore $\fid_{R} (M)$ satisfies the condition (2) in Lemma \ref{Mat24.2}, and we can conclude that $\fid_{R} (M)$ is an open subset of $\spec (R)$ by Lemma \ref{Mat24.2}.

(2): The assertion follows from (1).
\end{proof}

Note that the assumption
\begin{equation}\label{assump NC gor}
\gor(R/\p)\ {\rm contains\ a\ nonempty\ open\ subset\ of\ } \spec (R/\p)\ {\rm for\ all\ prime\ ideals\ } \p \ {\rm of\ } R.
\end{equation}
in Corollary \ref{corA} (2) is the same as that of (NC) for the Gorensteinness.
The above corollary yields the following result; it recovers theorems of Greco and Marinari and of Takahashi.

\begin{cor}\label{corA'}
\begin{enumerate}[\rm(1)]
\item {\rm (Greco--Marinari)} The Gorensteinness satisfies {\rm (NC)}.
\item Suppose that $R$ is an acceptable ring. 
Then $\fid_R(M)$ is open for all $R$-modules $M$.
\item If $R$ is a homomorphic image of a Gorenstein ring, then $\fid_R(M)$ is open for all $R$-modules $M$, and in particular, $\gor(R)$ is open.
\item {\rm (Takahashi)} If $R$ is excellent, then $\fid_R(M)$ is open for any $R$-module $M$.
\end{enumerate}
\end{cor}

\begin{proof}
The assertion (1) follows from Corollary \ref{corA} (2) because $\gor(R)=\fid_R(R)$. 
If $R$ is an acceptable ring, then (\ref{assump NC gor}) holds for $R$ by definition. 
Hence, Corollary \ref{corA} (2) gives the assertion (2).
Both a homomorphic image of a Gorenstein ring and an excellent ring are acceptable.
Thus the assertions (3) and (4) follow from (2).
\end{proof}

In the next section, the ring $R$ that satisfies (\ref{assump NC gor}) will further be studied.
We close this section by stating a proposition about (NC) for the Gorensteinness.

Put $(-)^\ast =\Hom_R (-,R)$.
Let $M$ be an $R$-module.
We say that $M$ is \textit{totally reflexive} if the natural homomorphism $M \to M^{\ast\ast}$ is isomorphic and $\Ext_R^i(M\oplus M^\ast,R)=0$ for all $i>0$.
Let $n\ge 0$ be an integer.
If $n$th syzygy of $M$ is totally reflexive, then we say that $M$ has \textit{G-dimension at most n}, and write $\operatorname{G-dim} (M)\le n$.
We say that $M$ has \textit{infinite G-dimension}, and write $\operatorname{G-dim} (M)=\infty$ if such an integer $n$ does not exist.

\begin{prop}\label{nc gor}
The following are equivalent.
\begin{enumerate}[\rm(1)]
\item $\gor(R)$ is open.
\item $\gor(R/\p)$ contains a nonempty open subset of $\spec (R/\p)$ for any $\p\in\gor(R)$.
\item $\fid_R(R/I)$ is open in $\spec (R)$ for any ideal $I$ of $R$.
\item $\fid_R(M)$ is open in $\spec (R)$ for any $R$-module $M$ such that $\operatorname{G-dim} (M)<\infty$.
\end{enumerate}
\end{prop}

\begin{proof}
The implications (3) $\Rightarrow$ (1) and (4) $\Rightarrow$ (1) hold since $\gor(R)=\fid_R(R)$.

(1) $\Rightarrow$ (2):\ 
For any $\p\in\gor(R)$, there exists $f\in R\setminus\p$ such that $\D(f)$ is contained in $\gor(R)$ because $\gor(R)$ is open.
The ring $R_f$ is Gorenstein.
Corollary \ref{corA'} (3) implies that $\gor((R/\p)_f)$ is an open subset of $\spec ((R/\p)_f)$.
This says that $\gor(R/\p)$ contains a nonempty open subset of $\spec (R/\p)$.

(2) $\Rightarrow$ (3):\ 
Let $I$ be an ideal of $R$ and let $\p\in\fid_R (R/I)\cap\supp_R(R/I)$.
Since $(R/I)_\p$ is a nonzero cyclic $R_\p$-module such that $\id_{R_\p}(R/I)_\p<\infty$,
it follows from \cite[Chapitre II, Th\'{e}or\`{e}me 5.5]{PS} that $R_\p$ is Gorenstein.
Thus, $\gor(R/\p)$ contains a nonempty open subset of $\spec (R/\p)$.
By Corollary \ref{corA} (1), $\fid_R(R/I)$ is open.

(2) $\Rightarrow$ (4):\ 
Let $M$ be a $R$-module $M$ such that $\operatorname{G-dim} (M)<\infty$, and let $\p\in\fid_R (M)\cap\supp_R(M)$.
Then $M_\p$ is a nonzero $R_\p$-module such that $\operatorname{G-dim}_{R_\p} (M_\p)<\infty$ and $\id_{R_\p}M_\p<\infty$.
It follows from \cite[Proposition 5.2.9]{C} and \cite[Corollary 3.3]{H} that $R_\p$ is Gorenstein.
A similar argument to the latter part of the proof of (2) $\Rightarrow$ (3) shows that $\fid_R(M)$ is open.
\end{proof}

%%%%%%%%%%%%%%%%%%%%%%%%%%%%%%%%%%%%%%%%%%%%%%%%%%%%%%%%%%%%%%%%%%%%
%%%%%%%%%%%%%%%%%%%%%%%%%%%%%%%%%%%%%%%%%%%%%%%%%%%%%%%%%%%%%%%%%%%%
\section{Some rings over which all $\fid$-loci are open}

In this section, we study the ring $R$ that satisfies (\ref{assump NC gor}).
The $\fid$-loci of all $R$-modules are open if $R$ is such a ring; see Corollary \ref{corA}.
The property (\ref{assump NC gor}) is stable under the following operations.

\begin{lem}\label{stable NC gor}
Suppose that {\rm (\ref{assump NC gor})} holds for $R$.
Then it also holds for
\begin{enumerate}[\rm(1)]
\item a homomorphic image of $R$,
\item a localization of $R$, and
\item the image of a Gorenstein homomorphism from $R$ (e.g., a polynomial ring over $R$).
\end{enumerate}
\end{lem}

\begin{proof}
(1): Let $I$ be an ideal of $R$, and $\q$ a prime ideal of $S=R/I$.
Then $\q=\p/I$ for some prime ideal $\p$ of $R$ containing $I$, and $S/\q\simeq R/\p$.
Therefore, $\gor(S/\q)$ contains a nonempty open subset of $\spec (S/\q)$.

(2): Let $S$ be a multiplicatively closed subset of $R$, and let $\p$ be a prime ideal of $R$ with $\p\cap S=\emptyset$.
Set $W=\{\q/\p\in\spec (R/\p)\mid \q\cap S=\emptyset\}$.
There is a natural homeomorphism $F:W\to\spec (R_S/\p R_S)$.
By assumption, $\gor(R/\p)$ contains a nonempty open subset of $\spec (R/\p)$.
It follows from Lemma \ref{base} (1) that $\gor(R/\p)\cap W$ contains a nonempty open subset of $W$.
Hence, $\gor(R_S/\p R_S)$ contains a nonempty open subset of $\spec (R_S/\p R_S)$
because $F(\gor(R/\p)\cap W)=\gor(R_S/\p R_S)$.

(3): Let $\phi:R\to S$ be a Gorenstein homomorphism. 
Let $\q$ be a prime ideal of $S$, and $\p=\q\cap R$.
Since $\gor(R/\p)$ contains a nonempty open subset of $\spec (R/\p)$, there exists $f\in R\setminus\p$ such that $(R/\p)_f$ is Gorenstein.
The induced ring homomorphism $\phi \otimes_R (R/\p)_f:(R/\p)_f \to S \otimes_R (R/\p)_f$ is also Gorenstein by \cite[Proposition 6.4]{Sh}.
It follows from \cite[Corollary 3.3.15]{BH} that $S \otimes_R (R/\p)_f$ is a Gorenstein ring.
There is a natural surjection $S \otimes_R (R/\p)_f\simeq (S/\p S)_{\phi(f)}\twoheadrightarrow (S/\q)_{\phi(f)}$.
By Corollary \ref{corA'} (3), we obtain $\gor((S/\q)_{\phi(f)})$ is open.
This says that $\gor(S/\q)$ contains a nonempty open subset of $\spec (S/\q)$.
\end{proof}

Let $S$ be a ring.
A ring homomorphism $\phi:R\to S$ is said to be \textit{essentially of finite type} if $S$ is the localization of a finitely generated $R$-algebra.
Below is a direct corollary of Lemma \ref{stable NC gor}.

\begin{cor}\label{stable NC gor essentially}
Let $S$ be a ring, and let $\phi:R\to S$ be a ring homomorphism.
Suppose that {\rm (\ref{assump NC gor})} holds for $R$, and $\phi$ is either essentially of finite type or Gorenstein.
Then {\rm (\ref{assump NC gor})} also holds for $S$.
In particular, $\fid_{S} (M)$ is open for any $S$-module $M$. 
\end{cor}

Suppose that $R$ is an Artinian ring.
Then {\rm (\ref{assump NC gor})} holds for $R$ since $R/\p$ is a field for any prime ideal $\p$ of $R$.
More generally, we can prove the result below.

\begin{prop}\label{gor fiber}
Let $S$ be a ring, and let $\phi:R\to S$ be either essentially of finite type or Gorenstein.
If $R$ is a well-fibered semi-local ring, then {\rm (\ref{assump NC gor})} holds for $S$, and thus $\fid_{S} (M)$ is open for any $S$-module $M$. 
In particular, $\fid_{S} (M)$ is open for any $S$-module $M$ if $R$ is an Artinian ring.
\end{prop}

\begin{proof}
It suffices to verify that {\rm (\ref{assump NC gor})} holds for $R$.
Let $\p$ be a prime ideal of $R$, and let $\m_1,\ldots,\m_r$ be all the maximal elements of $\V(\p)$.
Fix an integer $1\le i\le r$.
By assumption, the local ring $R_{\m_i}$ is well-fibered.
It follows from \cite[Corollary 6.5]{Sh} that $\gor(R_{\m_i}/\p R_{\m_i})$ is an open subset of $\spec (R_{\m_i}/\p R_{\m_i})$.
Put $U_i = \{\q/\p \in\spec (R/\p)\mid \q\subseteq\m_i \}$.
There is a natural homeomorphism $F:U_i \to \spec (R_{\m_i}/\p R_{\m_i})$.
Then $\gor(R/\p)\cap U_i$ is an open subset of $U_i$ since $F^{-1}(\gor(R_{\m_i}/\p R_{\m_i}))=\gor(R/\p)\cap U_i$. 
Hence there exists an open subset $V_i$ of $\spec (R/\p)$ such that $\gor(R/\p)\cap U_i=V_i\cap U_i$.
For any integer $1\le i\le r$, the zero ideal of $R/\p$ belongs to $V_i$.
Therefore, $V=\bigcap_{i=1}^r V_i$ is a nonempty open subset of $\spec (R/\p)$.
For any prime ideal $Q$ of $R/\p$ belonging to $V$, there exists an integer $1\le i\le r$ such that $Q$ is in $U_i$.
We have $Q\in V\cap U_i \subseteq V_i\cap U_i=\gor(R/\p)\cap U_i \subseteq\gor(R/\p)$.
\end{proof}

Now we prove the main result of this section.
It states that the converse of Theorem \ref{thmA} holds if a fixed prime ideal belongs to the $\mcm$-locus of a module.
The same techniques as in the proof of \cite[Theorem 1.4]{GM} play an essential role in the proof of the theorem.

\begin{thm}\label{gor}
Let $M$ be an $R$-module, and $\p\in\supp_R(M)\cap\fid_R(M)\cap\mcm_R(M)$.
The following conditions are equivalent.
\begin{enumerate}[\rm(1)]
\item $\fid_R(M)$ contains a nonempty open subset of $\V(\p)$.
\item $\gor(R/\p)$ contains a nonempty open subset of $\spec (R/\p)$.
\end{enumerate}
\end{thm}

\begin{proof}
Note that we can freely replace our ring $R$ with its localization $R_f$ for any element $f\in R\setminus\p$ to prove the theorem; we apply Lemma \ref{replace}.

First, we consider the case $\height\p=0$.
Then $M_\p$ is a nonzero injective $R_\p$-module because $\id_{R_\p}M_\p=\depth R_\p=0$ by \cite[Theorem 3.1.17]{BH}.
Hence, $\Ext_{R_\p}^1(\kappa(\p),M_\p)=0$ and $\Hom_{R_\p}(\kappa(\p),M_\p)\simeq \kappa(\p)^{\oplus n}$ for some integer $n>0$; here we set $\kappa(\p)=R_\p/\p R_\p$.
By (2) and (5) of Lemma \ref{base}, we may assume that the following conditions (i) and (ii) are satisfied.
On the other hand, since $\p$ is a minimal prime ideal of $R$, it follows from (4) and (6) of Lemma \ref{base} that we may assume that the following condition (iii) is satisfied.
\begin{enumerate}[\rm(i)]
\item $\Ext_R^1(R/\p,M)=0$.
\item $\Hom_R(R/\p,M)\simeq (R/\p)^{\oplus n}$ for some integer $n>0$.
\item There exists an integer $r>0$ such that $\p^r =0$ and $\p^{i}/\p^{i+1}$ is a free $R/\p$-module for each $1\le i\le r$.
\end{enumerate}
If $\Ext_R^j(R/\p,M)=0$ for an integer $j>0$, then it is seen that $\Ext_R^{j+1}(R/\p,M)\simeq\Ext_R^j(\p,M)=0$ by (iii).
Therefore, we have $\Ext_R^j(R/\p,M)=0$ for any $j>0$ by (i).
Let $I:0\to I^0\to I^1\to\cdots$ be an injective resolution of $M$.
It follows from \cite[Lemma 3.1.6]{BH} and (ii) that the complex $\Hom_R(R/\p,I)$ is an injective resolution of $\Hom_R(R/\p,M)\simeq (R/\p)^{\oplus n}$ as an $R/\p$-module, because $\Ext_R^j(R/\p,M)=0$ for any $j>0$.
Let $\q\in \V(\p)$. 
There is an isomorphism $\Hom_R(R/\q,I)\simeq\Hom_{R/\p}(R/\q,\Hom_R(R/\p,I))$ of complexes.
This says that $\Ext_R^j(R/\q,M)\simeq\Ext_{R/\p}^j(R/\q,R/\p)^{\oplus n}$ for any $j\ge0$.
We obtain an equivalence
\begin{equation}\label{fid ext}
\Ext_{R_\q}^j(\kappa(\q),M_\q)=0 \ \Leftrightarrow \ \Ext_{R_\q/\p R_\q}^j(\kappa(\q),R_\q/\p R_\q)=0
\end{equation}
for each integer $j\ge0$.
It follows from (\ref{fid ext}) and \cite[Proposition 3.1.14]{BH} that $\q$ is in $\fid_R(M)$ if and only if $\q/\p$ is in $\gor(R/\p)$.
Applying Lemma \ref{replace2} to $I=\p$ and $S=\fid_R(M)$, we see that the assertion holds.

Next, we address the general case.
Put $n=\height\p$.
It follows from \cite[Corollary 9.6.2, Remark 9.6.4(a)]{BH} that $R_\p$ is a Cohen--Macaulay local ring since $\id_{R_\p}M_\p<\infty$ and $M_\p\ne 0$.
On the other hand, $M_\p$ is a nonzero maximal Cohen--Macaulay $R_\p$-module.
By Lemma \ref{base} (3), we may assume that there exists a sequence $\bm{x}=x_1,\ldots,x_n$ of elements in $\p$, which is a regular sequence on $R$ and $M$.
Set $\overline{R}=R/\bm{x} R$, $\overline{\p}=\p/\bm{x} R$, and $\overline{M}=M/\bm{x} M$.
Then $\height{\overline{\p}}=0$ and $\overline{\p}$ belongs to $\supp_{\overline{R}} (\overline{M})$.
It follows from \cite[Corollary 3.1.15]{BH} that $\q$ is in $\fid_R(M)$ if and only if $\q/\bm{x} R$ is in $\fid_{\overline{R}}(\overline{M})$ for any $\q\in \V(\bm{x} R)$.
In particular, $\overline{\p}$ belongs to $\fid_{\overline{R}}(\overline{M})$.
Lemma \ref{replace2} implies that $\fid_R(M)$ contains a nonempty open subset of $\V(\p)$ if and only if $\fid_{\overline{R}}(\overline{M})$ contains a nonempty open subset of $\V(\overline{\p})$.
Applying the above argument in the case where $\height\p =0$,
the latter holds if and only if $\gor(\overline{R}/\overline{\p})=\gor(R/\p)$ contains a nonempty open subset of $\spec(\overline{R}/\overline{\p})=\spec(R/\p)$.
\end{proof}

The result below can be obtained from Theorem \ref{gor}.

\begin{cor}\label{Gmod}
Let $S$ be a ring, and let $\phi:R\to S$ be either essentially of finite type or Gorenstein.
Suppose that there exists an $R$-module $M$ such that $\supp_R(M)=\fid_R(M)=\mcm_R(M)=\spec (R)$.
Then {\rm (\ref{assump NC gor})} holds for $S$.
In particular, $\fid_S(N)$ is open for any $S$-module $N$.
\end{cor}

\begin{proof}
By Corollary \ref{stable NC gor essentially}, we have only to show that (\ref{assump NC gor}) holds for $R$.
Let $\p$ be a prime ideal of $R$.
Then $\p$ belongs to $\supp_R(M)\cap\fid_R(M)\cap\mcm_R(M)$.
Since $\V(\p)$ is contained in $\spec (R)=\fid_R(M)$, it follows from Theorem \ref{gor} that $\gor(R/\p)$ contains a nonempty open subset of $\spec (R/\p)$.
\end{proof}

Let $M$ be an $R$-module.
Let $R\ltimes M$ be the direct sum of $R$ and $M$.
Then $R\ltimes M$ has the ring structure with multiplication defined by $(a,x)(b,y)=(ab,ay+bx)$ for all $a,b\in R$ and $x,y\in M$.
The ring $R\ltimes M$ is called the \textit{idealization} of $M$ over $R$.
The natural surjection $R\ltimes M\to R$ is a ring homomorphism.
It is easy to see that $\spec(R\ltimes M)=\{\p\times M \mid \p\in\spec(R)\}$, and the map 
$$
(R\ltimes M)_{\p\times M} \to R_\p \ltimes M_\p, \quad
\frac{(a,x)}{(s,y)} \mapsto \left( \frac{a}{s} \raisebox{-1.57ex}{,\,\,} \frac{sx-ay}{s^2} \right)
\qquad (a\in R,\ s\in R\setminus\p,\ x,y\in M)
$$
is a ring isomorphism.
We say that $M$ is with full support if $\supp_R(M)=\spec (R)$.

\begin{rmk}
The assumption of Corollary \ref{Gmod} says that there is a Gorenstein $R$-module with full support.
Therefore, the above corollary recovers Corollary \ref{corA'} (3) because $R$ is itself a Gorenstein $R$-module with full support if $R$ is Gorenstein.
More generally, if $R$ is a Cohen--Macaulay ring and $M$ is a canonical module of $R$, then the assumption of Corollary \ref{Gmod} is also satisfied.
However, in this case, note that $R\ltimes M$ is Gorenstein; see the proof of \cite[Theorem 3.3.6]{BH}.
For example, Weston \cite{W} gives a Cohen--Macaulay local ring with a Gorenstein module of rank two (thus with full support), but without a canonical module.
\end{rmk}

%%%%%%%%%%%%%%%%%%%%%%%%%%%%%%%%%%%%%%%%%%%%%%%%%%%%%%%%%%%%%%%%%%%%
%%%%%%%%%%%%%%%%%%%%%%%%%%%%%%%%%%%%%%%%%%%%%%%%%%%%%%%%%%%%%%%%%%%%
\section{The openness of the $\cm$-locus of a module}

In this section, we study the openness of the $\cm$-locus and $\mcm$-locus of a module.
More precisely, we observe that the same results as we gave in the previous sections hold for the Cohen--Macaulay property.

Let $S$ be a ring.
A ring homomorphism $\phi:R\to S$ is said to be a \textit{Cohen--Macaulay homomorphism} if $\phi$ is flat and all the fiber rings of $\phi$ are Cohen--Macaulay \cite{Sha}.
We give the following condition.
\begin{equation}\label{assump NC cm}
\cm(R/\p)\ {\rm contains\ a\ nonempty\ open\ subset\ of\ } \spec (R/\p)\ {\rm 
for\ all\ prime\ ideals\ } \p \ {\rm of\ } R.
\end{equation}
The results below are the Cohen--Macaulay analogues of Lemma \ref{stable NC gor} and Corollary \ref{stable NC gor essentially}.

\begin{lem}\label{stable NC cm}
Suppose that {\rm (\ref{assump NC cm})} holds for $R$.
Then it also holds for
\begin{enumerate}[\rm(1)]
\item a homomorphic image of $R$,
\item a localization of $R$, and
\item the image of a Cohen--Macaulay homomorphism from $R$ (e.g., a polynomial ring over $R$).
\end{enumerate}
\end{lem}

\begin{proof}
The assertions can be shown in a similar way as in the proof of Lemma \ref{stable NC gor}.
However, note the following.
First, it is seen that the same assertion as \cite[Proposition 6.4]{Sh} holds for a Cohen-Macaulay homomorphism; replace \cite[Theorem 2.3]{Sh} with \cite[Theorem 2.1.10]{BH} in the proof of \cite[Proposition 6.4]{Sh}.
Second, there is a Cohen--Macaulay version of Corollary \ref{corA'} (3); see \cite[Exercise 24.2]{Mat}.
\end{proof}

\begin{cor}\label{stable NC cm essentially}
Let $S$ be a ring, and let $\phi:R\to S$ be a ring homomorphism.
Suppose that {\rm (\ref{assump NC cm})} holds for $R$, and $\phi$ is either essentially of finite type or Cohen--Macaulay.
Then {\rm (\ref{assump NC cm})} also holds for $S$.
\end{cor}

We give a lemma to state the main result of this section.

\begin{lem}\label{filtration}
Let $M$ be a nonzero $R$-module, $I$ an ideal of $R$, and $\bm{x}=x_1,\ldots,x_n$ a sequence of elements in $R$.
Suppose that there exists an integer $r>0$ such that $I^r M=0$, and $I^{i-1} M/I^i M$ is a free $R/I$-module for each $1\le i\le r$.
Then $\bm{x}$ is an $R/I$-regular sequence if and only if it is an $M$-regular sequence.
In particular, $\depth R_\p/I R_\p=\depth M_\p$ for any $\p\in\V(I)$.
\end{lem}

\begin{proof}
We prove the former assertion of the lemma by induction on $n$.
Suppose that $n=1$, and $x=x_1$.
We obtain $\ass_R (R/I)=\ass_R (M/I^i M)$ for all $1\le i\le r$ by induction on $i$.
In fact, the $R/I$-module $I^{i-1} M/I^i M$ is free and there is an exact sequence $0\to I^{i-1} M/I^i M\to M/I^i M\to M/I^{i-1} M\to 0$ for each $1\le i\le r$.
Therefore, for any $1\le i\le r$, the element $x$ is $R/I$-regular if and only if it is $M/I^i M$-regular.
In particular, if $x$ is such a regular element, then it follows from \cite[Proposition 1.1.4]{BH} that the sequence
$$
0\to (I^{i-1} M/I^i M)\otimes_R (R/xR)\to (M/I^i M)\otimes_R (R/xR)\to (M/I^{i-1} M)\otimes_R (R/xR)\to 0
$$ 
is exact for any $1\le i\le r$. 
We see that 
$$
(R/I)\otimes_R (R/xR)\ne 0 \ \Leftrightarrow \  (M/I M)\otimes_R (R/xR)\ne 0\Leftrightarrow \cdots \Leftrightarrow \ M\otimes_R (R/xR)\ne 0
$$ 
by induction because $I^{i-1} M/I^i M$ is a free $R/I$-module for all $1\le i\le r$.
Thus the assertion follows.

Suppose that $n>1$.
We may assume that $\bm{x}'=x_1,\ldots,x_{n-1}$ is a regular sequence on $R/I$, $M/I M, \cdots, M/I^{r-1} M$, and $M$.
Set $\overline{R}=R/\bm{x}' R$, $\overline{I}=I \overline{R}$, and $\overline{M}=M/\bm{x}' M$.
Then $\overline{I}^r \overline{M}=0$, and $\overline{I}^{i-1} \overline{M}/\overline{I}^i \overline{M}$ is a free $\overline{R}/\overline{I}$-module for each $1\le i\le r$ because there is a commutative diagram
\[
  \xymatrix@C=20pt@R=15pt{
    0 \ar[r]
    & (I^{i-1} M/I^i M)\otimes_R \overline{R} \ar[r]
    & (M/I^i M)\otimes_R \overline{R} \ar[r] \ar[d]_{\rotatebox{90}{$\sim$}} 
    & (M/I^{i-1} M)\otimes_R \overline{R} \ar[r] \ar[d]_{\rotatebox{90}{$\sim$}} 
    & 0 \\
    & 
    & (R/I^i)\otimes_R M\otimes_R \overline{R} \ar[d]_{\rotatebox{90}{$\sim$}}
    & (R/I^{i-1})\otimes_R M\otimes_R \overline{R} \ar[d]_{\rotatebox{90}{$\sim$}}
    & \\
    0 \ar[r]
    & \overline{I}^{i-1} \overline{M}/\overline{I}^i \overline{M} \ar[r]
    & \overline{M}/\overline{I}^i \overline{M} \ar[r]  
    & \overline{M}/\overline{I}^{i-1} \overline{M} \ar[r] 
    & 0 \\
  }
\]
with exact rows by \cite[Proposition 1.1.4]{BH}.
Applying the case $n=1$ shows the assertion.

Next we prove the latter assertion of the lemma.
Let $\p$ be a prime ideal of $R$ containing $I$.
Since $M/IM$ is a nonzero free $R/I$-module, we obtain $M_\p\ne 0$.
It is seen that $I^r M_\p=0$, and $I^{i-1} M_\p/I^i M_\p$ is a free $R_\p/I R_\p$-module for each $1\le i\le r$.
The assertion follows from the former.
\end{proof}

Now we can prove the main result of this section by using \cite[Theorem 24.5]{Mat}.
It corresponds to Theorems \ref{thmA} and \ref{gor}.

\begin{thm}\label{cm}
Let $M$ be an $R$-module, and $\p\in\supp_R(M)\cap\cm_R(M)$.
Then the following two conditions are equivalent.
\begin{enumerate}[\rm(1)]
\item $\cm_R(M)$ contains a nonempty open subset of $\V(\p)$.
\item $\cm(R/\p)$ contains a nonempty open subset of $\spec (R/\p)$.
\end{enumerate}
In addition, if $\p$ belongs to $\mcm_R(M)$, then the following is also equivalent.
\begin{enumerate}[\rm(3)]
\item $\mcm_R(M)$ contains a nonempty open subset of $\V(\p)$.
\end{enumerate}
\end{thm}

\begin{proof}
First of all, we can freely replace our ring $R$ with its localization $R_f$ for any element $f\in R\setminus\p$ to prove the theorem; we apply Lemma \ref{replace}.

(i): We consider the case where $\p$ belongs to $\mcm_R(M)$. We prove the equivalence (2) $\Leftrightarrow$ (3).

First, we deal with the case $\height\p=0$.
Since $\p$ is a minimal prime ideal of $R$, we may assume that $\p^r=0$ for some integer $r>0$, and $\p^{i-1} M/\p^i M$ is a free $R/\p$-module for each $1\le i\le r$ by (4) and (6) of Lemma \ref{base}.
Fix $\q\in \V(\p)$.
Note that $M_\q\ne 0$.
Lemma \ref{filtration} implies that
\begin{equation}\label{cm1}
\depth R_\q/\p R\q=\depth M_\q .
\end{equation}
On the other hand, $\p$ is the only minimal prime ideal of $R$ since $\p^r=0$.
Thus, we obtain
\begin{equation}\label{cm2}
\dim R_\q/\p R\q=\dim R_\q .
\end{equation}
It follows from (\ref{cm1}) and (\ref{cm2}) that $\q/\p$ is in $\cm(R/\p)$ if and only if $\q$ is in $\mcm_R(M)$.
Applying Lemma \ref{replace2} to $I=\p$ and $S=\mcm_R(M)$, we see that the equivalence (2) $\Leftrightarrow$ (3) holds.

Next, we handle the general case.
Put $n=\depth M_\p=\dim R_\p$.
We may assume that there exists a sequence $\bm{x}=x_1,\ldots,x_n$ of elements in $\p$ such that $\bm{x}$ is an $M$-regular sequence and $\height\bm{x}R=\height\p=n$ by Lemma \ref{base} (3). 
Fix $\q\in \V(\p)$.
We have
$$
\depth M_\q /\bm{x} M_\q=\depth M_\q-n,\quad
\dim R_\q /\bm{x} R_\q=\dim R_\q-n.
$$
It is seen that $\q$ is in $\mcm_R(M)$ if and only if $\q /\bm{x} R$ is in $\mcm_{R/\bm{x} R}(M/\bm{x} M)$.
A similar argument to the latter part of the proof of Theorem \ref{gor} shows the equivalence (2) $\Leftrightarrow$ (3).

(ii): Finally, we prove that the equivalence (1) $\Leftrightarrow$ (2) holds if $\p$ belongs to $\supp_R(M)\cap\cm_R(M)$.
Put $I=\operatorname{Ann}_R (M)$, and $\overline{R}=R/I$.
Then $\cm_{\overline{R}} (M)=\mcm_{\overline{R}} (M)$ since $\supp_{\overline{R}} (M)=\spec (\overline{R})$.
It is easy to see that $\q$ is in $\cm_R (M)$ if and only if $\q/I$ is in $\mcm_{\overline{R}} (M)$ for any $\q\in \V(I)$.
In particular, $\p/I$ belongs to $\mcm_{\overline{R}} (M)$.
An analogous argument to the latter part of the proof of Theorem \ref{gor} shows the equivalence (1) $\Leftrightarrow$ (2); we apply Lemma \ref{replace2} and the case (i).
\end{proof}

The following corollaries can be shown similarly as in the proof of Corollaries \ref{corA} and \ref{Gmod}, respectively.

\begin{cor}\label{cor cm}
Let $M$ be an $R$-module.
\begin{enumerate}[\rm(1)]
\item Suppose that $\cm(R/\p)$ contains a nonempty open subset of $\spec (R/\p)$ for any $\p\in\cm_R(M)\cap\supp_R(M)$.
Then $\cm_R(M)$ is an open subset of $\spec (R)$.
\item Suppose that $\cm(R/\p)$ contains a nonempty open subset of $\spec (R/\p)$ for any $\p\in\mcm_R(M)\cap\supp_R(M)$.
Then $\mcm_R(M)$ is an open subset of $\spec (R)$.
\item In particular, if {\rm (\ref{assump NC cm})} holds for $R$, then $\cm_R(M)$ and $\mcm_R(M)$ are open.
\end{enumerate}
\end{cor}

\begin{cor}\label{CMmod}
\begin{enumerate}[\rm(1)]
\item Let $S$ be a ring, and let $\phi:R\to S$ be either essentially of finite type or Cohen--Macaulay.
Suppose that there exists an $R$-module $M$ such that $\supp_R(M)=\cm_R(M)=\spec (R)$.
Then {\rm (\ref{assump NC cm})} holds for $S$, and thus $\cm_S(N)$ and $\mcm_S(N)$ are open for any $S$-module $N$.
\item In particular, the $\cm$ and $\mcm$-loci of a module over a homomorphic image of a Cohen--Macaulay ring are open.
\end{enumerate}
\end{cor}

The result below is a Cohen--Macaulay version of Proposition \ref{gor fiber}.
It recovers \cite[Remark 3.4]{DJ}.

\begin{prop}\label{cm fiber}
Let $S$ be a ring, and let $\phi:R\to S$ be either essentially of finite type or Cohen--Macaulay.
The loci $\cm_{S} (M)$ and $\mcm_{S} (M)$ are open for any $S$-module $M$ if $R$ is semi-local and the natural ring homomorphism from $R_\p$ to its completion is Cohen--Macaulay for every prime ideal $\p$ of $R$.
In particular, if $R$ is an Artinian ring, then $\cm_{S} (M)$ and $\mcm_{S} (M)$ are open for any $S$-module $M$.
\end{prop}

\begin{proof}
Using the proof of \cite[Theorems 3.3 and 6.4]{Sh}, we see that there is a Cohen--Macaulay version of \cite[Corollary 6.5]{Sh}; see \cite[Theorems 2.1.7 and 2.1.10]{BH}, and \cite[Exercise 24.2]{Mat}.
Hence, the proposition is shown analogously as in the proof of Proposition \ref{gor fiber}.
\end{proof}

%%%%%%%%%%%%%%%%%%%%%%%%%%%%%%%%%%%%%%%%%%%%%%%%%%%%%%%%%%%%%%%%%%%%
%%%%%%%%%%%%%%%%%%%%%%%%%%%%%%%%%%%%%%%%%%%%%%%%%%%%%%%%%%%%%%%%%%%%
\section{The openness of the $(\S_n)$-locus of a module}

In this section, we study the openness of the $(\S_n)$-locus and $(\T_n)$-locus of a module.
The main result of this section is the theorem below.
Note that the assumptions of the theorem are weaker than the condition (2) in Lemma \ref{Mat24.2}.

\begin{thm}\label{Sn open}
Let $M$ be an $R$-module, and let $n\ge 0$ be an integer.
\begin{enumerate}[\rm(1)]
\item Suppose that $\S_n^R(M)$ contains a nonempty open subset of $\V(\p)$ for all $\p\in\supp_R(M)\cap\S_n^R(M)$ such that $\height\p<n$.
Then $\S_n^R(M)$ is an open subset of $\spec (R)$.
\item Suppose that $\T_n^R(M)$ contains a nonempty open subset of $\V(\p)$ for all $\p\in\supp_R(M)\cap\T_n^R(M)$ such that $\dim M_\p<n$.
Then $\T_n^R(M)$ is an open subset of $\spec (R)$.
\end{enumerate}
\end{thm}

\begin{proof}
(1): Put $X=\spec (R)\setminus\S_n^R(M)$. 
Let $\overline{X}$ be the closure of $X$.
It suffices to show that $\overline{X}=X$.
There exists an ideal $I$ of $R$ such that $\overline{X}=\V(I)$ since $\overline{X}$ is closed.
We may assume $I\ne R$.
\begin{spacing}{1.2}
\ \textbf{Claim.} \ If $\p$ is a minimal prime ideal of $I$, then $\p$ belongs to $X$.
\end{spacing}
\noindent \textit{Proof of Claim.}
We prove the claim by contradiction.
Suppose that there exists a minimal prime ideal $\p$ of $I$, which belongs to $\S_n^R(M)$.

(i): Suppose either $\p\notin\supp_R(M)$ or $\height\p<n$.
In these cases, $\S_n^R(M)$ contains a nonempty open subset of $\V(\p)$; see Remark \ref{rmk2} and the assumption of (1).
We can choose $f\in R\setminus\p$ such that $\D(f)\cap\V(\p)$ is contained in $\S_n^R(M)$ by Lemma \ref{base} (1).
Furthermore, it follows from Lemma \ref{base} (4) that we may assume $\p R_f=\sqrt{I R_f}$.
Now $\p$ is in $\V(I)=\overline{X}$, and $\D(f)$ is a neighbourhood of $\p$.
Hence $\D(f)\cap X$ is nonempty.
On the other hand, $\D(f)\cap X$ is contained in $\D(f)\cap\V(I)=\D(f)\cap\V(\p)$ since $\p R_f=\sqrt{I R_f}$.
It is seen that $\D(f)\cap X$ is contained in $\S_n^R(M)\cap X=\emptyset$.
This is a contradiction.

(ii): Suppose that $\p\in\supp_R(M)$ and $\height\p\ge n$.
Since $\p$ belongs to $\S_n^R(M)$ and $\height\p\ge n$, we have $\depth M_\p \ge {\rm inf}\{n, \height\p \}=n$.
Now $M_\p\ne0$.
It follows from (3) and (4) of Lemma \ref{base} that there exist a sequence $\bm{x}=x_1,\ldots,x_n$ of elements in $\p$ and $f\in R\setminus\p$ such that $\bm{x}$ is an $M_f$-regular sequence and $\p R_f=\sqrt{I R_f}$.
Similarly as in the proof of (i), $\D(f)\cap X$ is nonempty.
Let $\q\in\D(f)\cap X$.
There exists a prime ideal $\q'$ of $R$ such that $\q'\subseteq\q$ and $\depth M_{\q'} < {\rm inf}\{n, \height\q' \}$ because $\q$ is not in $\S_n^R(M)$.
It is easy to see that $\q'$ belongs to the subset $\D(f)\cap X$ of $\D(f)\cap\V(I)=\D(f)\cap\V(\p)$.
We obtain $\depth M_{\q'}\ge\operatorname{grade}(\q' R_f, M_f)\ge n$ since $\bm{x}$ is contained in $\q'$.
This is a contradiction.
\begin{spacing}{1.2}
\end{spacing}
Let $\p_1,\ldots,\p_r$ be all the minimal prime ideals of $I$.
The locus $\S_n^R(M)$ is stable under generalization.
Therefore, we have $\overline{X}=\V(I)=\bigcup_{i=1}^r \V(\p_i)\subseteq X$.

(2): Put $I=\operatorname{Ann}_R (M)$, and $\overline{R}=R/I$.
It is easy to see that $\dim M_\p=\height(\p/I)$, and $\p$ is in $\T_n^R(M)$ if and only if $\p/I$ is in $\S_n^{\overline{R}}(M)$ for any $\p\in \V(I)$.
It follows from Lemma \ref{replace2} and the assertion (1) that $\S_n^{\overline{R}}(M)$ is an open subset of $\spec(\overline{R})$.
Then $\T_n^R(M)\cap\V(I)$ is an open subset of $\V(I)$ since there is a natural homeomorphism $\V(I)\to\spec(\overline{R})$.
There exists an open subset $U$ of $\spec(R)$ such that $\T_n^R(M)\cap\V(I)=U\cap\V(I)$.
We easily obtain $\T_n^R(M)=U\cup\D(I)$.
\end{proof}

The following result is a corollary of the above theorem.

\begin{cor}\label{cm to Sn}
Let $M$ be an $R$-module, and let $n\ge 0$ be an integer.
\begin{enumerate}[\rm(1)]
\item If $\mcm_R(M)$ is open, then so is $\S_n^R(M)$.
\item If $\cm_R(M)$ is open, then so is $\T_n^R(M)$.
\end{enumerate}
\end{cor}

\begin{proof}
Suppose that $\mcm_R(M)$ is open.
Let $\p\in\supp_R(M)\cap\S_n^R(M)$ such that $\height\p<n$.
Then $\p$ belongs to $\mcm_R(M)$ since $\depth M_\p \ge {\rm inf}\{n, \height\p \}=\height\p$.
Therefore, $\mcm_R(M)\cap\V(\p)$ is a nonempty open subset of $\V(\p)$, and it is contained in $\S_n^R(M)$.
The assertion (1) follows from Theorem \ref{Sn open} (1).
The assertion (2) can be shown in a similar way.
\end{proof}

\begin{rmk}
In Section 5, we studied some rings over which the $\cm$ and $\mcm$-loci of all modules are open.
Corollary \ref{cm to Sn} says that $\S_n^R(M)$ and $\T_n^R(M)$ are open for any $R$-module $M$ if $R$ is such a ring.
\end{rmk}

Now we can prove the result below by applying \cite[Theorem 2.2]{Ta}.

\begin{thm}\label{Sn nc}
Let $M$ be an $R$-module, $\p$ a prime ideal of $R$, and let $n\ge 0$ be an integer.
Suppose that $\S_n(R/\p)$ contains a nonempty open subset of $\spec(R)$.
\begin{enumerate}[\rm(1)]
\item If $\p$ belongs to $\S_n^R(M)$ and $\height\p\le n$, then $\S_n^R(M)$ contains a nonempty open subset of $\V(\p)$.
\item If $\p$ belongs to $\T_n^R(M)$ and $\dim M_\p\le n$, then $\T_n^R(M)$ contains a nonempty open subset of $\V(\p)$.
\end{enumerate}
\end{thm}

\begin{proof}
First of all, by Remark \ref{rmk2}, we may assume that $\p$ is in $\supp_R(M)$.
Also, we can freely replace our ring $R$ with its localization $R_f$ for an element $f\in R\setminus\p$ to prove the theorem; we apply Lemma \ref{replace}.

(1): Since $\p$ belongs to $\S_n^R(M)$ and $\height\p\le n$, we have $\depth M_\p \ge {\rm inf}\{n, \height\p \}=\height\p$.
Put $d=\height\p$.
Note that $M_\p\ne0$.
It follows from (3), (4), and (6) of Lemma \ref{base} that we may assume that the following three conditions are satisfied.
\begin{enumerate}[\rm(i)]
\item There exist a sequence $\bm{x}=x_1,\ldots,x_d$ in $\p$ such that $\bm{x}$ is an $M$-regular sequence, and $\p=\sqrt{\bm{x} R}$.
\item Set $\overline{R}=R/\bm{x} R$, $\overline{\p}=\p/\bm{x} R$, and $\overline{M}=M/\bm{x} M$.
Then there exists an integer $r>0$ such that $\overline{\p}^r =0$ and $\overline{\p}^{i}\overline{M}/\overline{\p}^{i+1}\overline{M}$ is a free $R/\p$-module for each $1\le i\le r$.
\item $R/\p$ satisfies $(\S_n)$.
\end{enumerate}
Now we claim that $\V(\p)$ is contained in $\S_n^R(M)$.
Suppose that $\p'\in\V(\p)$ and $\q\in\spec(R)$ with $\q\subseteq\p'$, then we prove that $\depth M_\q \ge {\rm inf}\{n, \height\q \}$.

First, we deal with the case $\height(\q+\p/\p)\le n$.
Since $\q+\p$ is contained in $\p'$, we can choose $\q'\in\V(\q+\p)$ such that $\height(\q'/\p)=\height(\q+\p/\p)$.
It follows from the above three conditions and Lemma \ref{filtration} that
$$
\depth M_{\q'}-d=\depth \overline{M}_{\q'}=\depth R_{\q'}/\p R_{\q'}\ge {\rm inf}\{n, \height(\q'/\p) \}=\height(\q'/\p)=\height(\q'/\bm{x} R)=\height\q'-d.
$$
This says that $M_{\q'}$ is maximal Cohen--Macaulay, and so is $M_\q$.
We obtain $\depth M_\q \ge\height\q\ge {\rm inf}\{n, \height\q \}$.

Next, we consider the case $\height(\q+\p/\p)> n$.
It follows from (iii) and \cite[Proposition 1.2.10 (a)]{BH} that $\operatorname{grade}(\q+\p/\p, R/\p)={\rm inf}\{\depth R_{\q'}/\p R_{\q'} \mid \q'/\p\in\V(\q+\p/\p)\}\ge n$.
Hence, there exists an $R/\p$-regular sequence $\bm{y}=y_1,\ldots,y_n$ in $\q$.
It is also an $\overline{M}$-regular sequence by (ii) and Lemma \ref{filtration}.
By (i), the sequence $\bm{x}, \bm{y}=x_1,\ldots,x_d, y_1,\ldots,y_n$ in $\p'$ is an $M_{\p'}$-regular sequence, and so is $\bm{y}$.
Therefore, $\bm{y}$ is an $M_\q$-regular sequence.
We obtain $\depth M_\q \ge n\ge {\rm inf}\{n, \height\q \}$.

(2): Put $I=\operatorname{Ann}_R (M)$, and $\overline{R}=R/I$.
We see that $\dim M_\q=\height(\q/I)$, and $\q$ is in $\T_n^R(M)$ if and only if $\q/I$ is in $\S_n^{\overline{R}}(M)$ for any $\q\in \V(I)$.
It follows from the assertion (1) and Lemma \ref{replace2} that $\T_n^R(M)$ contains a nonempty open subset of $\V(\p)$.
\end{proof}

The same result as Corollary \ref{cor cm} holds for Serre's condition $(\S_n)$.

\begin{cor}\label{cor of Sn}
Let $M$ be an $R$-module, and let $n\ge 0$ be an integer.
\begin{enumerate}[\rm(1)]
\item Suppose that $\S_n(R/\p)$ contains a nonempty open subset of $\spec (R/\p)$ for any $\p\in\S_n^R(M)\cap\supp_R(M)$ such that $\height\p<n$.
Then $\S_n^R(M)$ is open.
\item Suppose that $\S_n(R/\p)$ contains a nonempty open subset of $\spec (R/\p)$ for any $\p\in\T_n^R(M)\cap\supp_R(M)$ such that $\dim M_\p<n$.
Then $\T_n^R(M)$ is open.
\item If $\S_n(R/\p)$ contains a nonempty open subset of $\spec (R/\p)$ for all prime ideals $\p$ of $R$, then $\S_n^R(M)$ and $\T_n^R(M)$ are open.
\end{enumerate}
\end{cor}

\begin{proof}
The assertions follow from Theorems \ref{Sn open} and \ref{Sn nc}.
\end{proof}

Let $n\ge 0$ be an integer.
We give the following condition.
$$
(\S_n)^\dag : \S_n(R/\p)\ {\rm contains\ a\ nonempty\ open\ subset\ of\ } \spec (R/\p)\ {\rm for\ all\ prime\ ideals\ } \p \ {\rm of\ } R.
$$

\begin{prop}\label{stable NC cm}
Let $m, n\ge 0$ be integers.
Suppose that $(\S_n)^\dag$ holds for $R$.
\begin{enumerate}[\rm(1)]
\item If $S$ is either a homomorphic image of $R$ or a localization of $R$, then $(\S_n)^\dag$ also holds for $S$.
\item If $S=R[X_1,\ldots,X_m]$ is a polynomial ring over $R$, then $(\S_{n-m})^\dag$ holds for $S$.
\end{enumerate}
\end{prop}

\begin{proof}
(1): The assertion can be shown analogously as in the proof of Lemma \ref{stable NC gor}.

(2): We may assume $m=1$, and $X_1=X$.
Let $\q$ be a prime ideal of $S$, and $\p=\q\cap R$.
By \cite[Theorem 15.5]{Mat}, we have $\height(\q/\p S)\le 1$.
We prove that $\S_{n-1}(S/\q)$ contains a nonempty open subset of $\spec(S/\q)$.

First, we consider the case $\height(\q/\p S)=0$.
There exists $f\in R\setminus\p$ such that $(R/\p)_f$ satisfies $(\S_n)$ because $\S_n(R/\p)$ contains a nonempty open subset of $\spec (R/\p)$.
Then $(S/\q)_f\simeq (R/\p)_f[X]$ is a polynomial ring over $(R/\p)_f$.
It follows from \cite[Proposition 2.1.16]{BH} that $(S/\q)_f$ satisfies $(\S_n)$.
This says that $\S_n(S/\q)$ contains a nonempty open subset of $\spec(S/\q)$.

Next, we handle the case $\height(\q/\p S)=1$.
We can take an element $x$ of $\q\setminus\p S$.
It follows from (4) and (6) of Lemma \ref{base}, and Lemma \ref{filtration} that we can choose $g\in S\setminus\q$ such that $\depth (S/\q)_{\q'}=\depth (S/xS+\p S)_{\q'}$ for any $\q'\in\V(\q)\cap\D(g)$.
On the other hand, there exists $f\in R\setminus\p$ such that $(S/\p S)_f$ satisfies $(\S_n)$; see the above case.
For any prime ideal $\q'$ of $S$ which belongs to $\V(\q)\cap\D(g)\cap\D(f)$,
we obtain
$$
\depth (S/\q)_{\q'}=\depth (S/\p S)_{\q'}-1
\ge {\rm inf}\{n, \height(\q'/\p S) \}-1={\rm inf}\{n-1, \height(\q'/\q) \}.
$$
This says that $(S/\q)_{fg}$ satisfies $(\S_{n-1})$.
Thus, $\S_{n-1}(S/\q)$ contains a nonempty open subset of $\spec(S/\q)$ since $fg$ is in $S\setminus\q$.
\end{proof}

%%%%%%%%%%%%%%%%%%%%%%%%%%%%%%%%%%%%%%%%%%%%%%%%%%%%%%%%%%%%%%%%%%%%
%%%%%%%%%%%%%%%%%%%%%%%%%%%%%%%%%%%%%%%%%%%%%%%%%%%%%%%%%%%%%%%%%%%%
\section{Nagata criterion for module properties}

In this section, we prove that the statement (NC)$^{\ast}$, which was defined in Section 1, holds for the finite injective dimension property, the Gorenstein property, the Cohen--Macaulay property, the maximal Cohen--Macaulay property, and Serre's conditions $(\S_n)$ and $(\T_n)$. 

\begin{lem}\label{mod NC gor}
Let $M$ be an $R$-module, and let $\p\in\supp_R(M)$. 
The following are equivalent.
\begin{enumerate}[\rm(1)]
\item $\gor(R/\p)$ contains a nonempty open subset of $\spec (R/\p)$.
\item $\fid_{R/\p}(M/\p M)$ contains a nonempty open subset of $\spec (R/\p)$.
\item $\gor_{R/\p}(M/\p M)$ contains a nonempty open subset of $\spec (R/\p)$.
\end{enumerate}
\end{lem}

\begin{proof}
We obtain $\supp_R (M/\p M)=\supp_R(M)\cap V(\p)=V(\p)$, and thus $\supp_{R/\p} (M/\p M)=\spec (R/\p)$.
Hence, we have
\begin{align}\label{equal in free}
\gor(R/\p)\cap\free_{R/\p}(M/\p M)&=\fid_{R/\p} (M/\p M)\cap\free_{R/\p}(M/\p M) \\
&=\gor_{R/\p} (M/\p M)\cap\free_{R/\p}(M/\p M). \notag
\end{align}
By \cite[Theorem 4.10 (ii)]{Mat}, $\free_{R/\p}(M/\p M)$ is a nonempty open subset of $\spec (R/\p)$.
The equivalence follows from Lemma \ref{base} (1) and (\ref{equal in free}).
\end{proof}

The lemmas below can be shown along the same lines as in the proof of Lemma \ref{mod NC gor}.

\begin{lem}\label{mod NC cm}
Let $M$ be an $R$-module, and let $\p\in\supp_R(M)$. 
The following are equivalent.
\begin{enumerate}[\rm(1)]
\item $\cm(R/\p)$ contains a nonempty open subset of $\spec (R/\p)$.
\item $\cm_{R/\p}(M/\p M)$ contains a nonempty open subset of $\spec (R/\p)$.
\item $\mcm_{R/\p}(M/\p M)$ contains a nonempty open subset of $\spec (R/\p)$.
\end{enumerate}
\end{lem}

\begin{lem}\label{mod NC Sn}
Let $M$ be an $R$-module, and let $\p\in\supp_R(M)$. 
The following are equivalent.
\begin{enumerate}[\rm(1)]
\item $\S_n(R/\p)$ contains a nonempty open subset of $\spec (R/\p)$.
\item $\S_n^{R/\p}(M/\p M)$ contains a nonempty open subset of $\spec (R/\p)$.
\item $\T_n^{R/\p}(M/\p M)$ contains a nonempty open subset of $\spec (R/\p)$.
\end{enumerate}
\end{lem}

\begin{rmk}
Note that the conditions (2) and (3) of the above lemmas always hold when $\p$ is not in $\supp_R(M)$; see Remark \ref{rmk2}.
\end{rmk}

The same result as Theorem \ref{gor} holds for the Gorenstein property.

\begin{prop}\label{gor2}
Let $M$ be an $R$-module, and $\p\in\supp_R(M)\cap\gor_R(M)$.
The following are equivalent.
\begin{enumerate}[\rm(1)]
\item $\gor_R(M)$ contains a nonempty open subset of $\V(\p)$.
\item $\gor(R/\p)$ contains a nonempty open subset of $\spec (R/\p)$.
\end{enumerate} 
\end{prop}

\begin{proof}
(1) $\Rightarrow$ (2): Since $\gor_R(M)$ is contained in $\fid_R(M)$,
$\fid_R(M)$ contains a nonempty open subset of $\V(\p)$.
Theorem \ref{gor} yields that the condition (2) is satisfied.

(2) $\Rightarrow$ (1): Theorem \ref{gor} implies that $\fid_R(M)$ contains a nonempty open subset $U$ of $\V(\p)$.
On the other hand, $\cm(R/\p)$ contains a nonempty open subset of $\spec (R/\p)$ because $\cm(R/\p)$ contains $\gor(R/\p)$.
It follows from Theorem \ref{cm} that there exists a nonempty open subset $V$ of $\V(\p)$, which is contained in $\mcm_R(M)$.
Now $\p$ belongs to $U\cap V$ by Lemma \ref{base} (1).
Thus $U\cap V$ is a nonempty open subset of $\V(\p)$, and it is contained in $\fid_R(M)\cap\mcm_R(M)=\gor_R(M)$.
\end{proof}

\begin{cor}\label{corB}
\begin{enumerate}[\rm(1)]
\item Let $M$ be an $R$-module.
Suppose that $\gor(R/\p)$ contains a nonempty open subset of $\spec (R/\p)$ for any $\p\in\gor_R (M)\cap\supp_R(M)$.
Then $\gor_{R} (M)$ is an open subset of $\spec (R)$.
\item Suppose that $\gor(R/\p)$ contains a nonempty open subset of $\spec (R/\p)$ for all prime ideals $\p$ of $R$.
Then $\gor_{R} (M)$ is an open subset of $\spec (R)$ for any $R$-module $M$. 
\end{enumerate}
\end{cor}

\begin{proof}
It can be shown in a similar way as in the proof of Corollary \ref{corA}.
\end{proof}

Proposition \ref{gor2} can also be shown analogously as in the proof of Theorem \ref{gor}. 
In particular, Leuschke \cite{L} proved the same result as Corollary \ref{corB} using the same methods as in the proof of Theorem \ref{gor}.

From the above lemmas and corollaries, we can prove the main result of this section.

\begin{thm}\label{(NC)*}
Let $n\ge 0$ be an integer.
Then {\rm(NC)$^{\ast}$} holds for each $\P\in\{\fid, \gor, \cm, \mcm, (\S_n), (\T_n)\}$.
\end{thm}

\begin{proof}
The asssertion follows from Lemma \ref{mod NC gor} and Corollary \ref{corA} for $\P=\fid$, from Lemma \ref{mod NC gor} and Corollary \ref{corB} for $\P=\gor$, from Lemma \ref{mod NC cm} and Corollary \ref{cor cm} for $\P\in\{\cm, \mcm\}$, and from Lemma \ref{mod NC Sn} and Corollary \ref{cor of Sn} for $\P\in\{(\S_n), (\T_n)\}$.
\end{proof}

%%%%%%%%%%%%%%%%%%%%%%%%%%%%%%%%%%%%%%%%%%%%%%
\begin{ac}
The author would like to thank his supervisor Ryo Takahashi for valuable comments.
\end{ac}
%%%%%%%%%%%%%%%%%%%%%%%%%%%%%%%%%%%%%%%%%%%%%%%%%%%%%%%%%%%%%

\end{document}